\documentclass{article}

\usepackage{amssymb,amsthm}
\usepackage{amsmath}
\usepackage{latexsym}
\usepackage{amssymb}
\usepackage[dvips]{graphics}
\usepackage[dvips]{graphicx}
\newtheorem{theorem}{Theorem}[section]

\newtheorem{lemma}[theorem]{Lemma}
\newtheorem{corollary}[theorem]{Corollary}
\newtheorem{definition}[theorem]{Definition}
\newtheorem{remark}[theorem]{Remark}
\newtheorem{example}[theorem]{Example}

\newcommand{\R}{{\mathbf R}}

\renewcommand{\int}{\rm Int}
\newcommand{\E}{\mathbb E}

\newcommand{\PP}{{\rm P}}
\newcommand{\p}{{\mathfrak p}}

\newcommand{\st}{\rm {St}}
\newcommand{\lk}{\rm {lk}}

\begin{document}

\title{Large random simplicial complexes, I}        
\author{A. Costa and M. Farber}        
\date{March 18, 2015}          
\maketitle

\section{Introduction}
Networks, such as the Internet as well as social and biological networks of various nature, have been the subject of intense study in recent years. 
Usually one thinks of a network as being a large graph with nodes representing objects (or sites) and edges representing connections or links between the objects. Random graph theory provides a language and mathematical techniques for studying large random networks in different contexts.

If we are interested not only in pairwise relations between the objects but also in relations between triples, quadruples etc, we will have to use instead of graphs the high dimensional simplicial complexes as geometrical models of networks. 
The mathematical study of large random simplicial complexes started relatively recently and several different probabilistic models of random topological objects have appeared within the last 10 years, see \cite{CFK} and \cite{Ksurvey} for surveys. 
One may mention random surfaces \cite{PS}, random 3-dimensional manifolds
\cite{DT}, random configuration spaces of linkages 
 \cite{F}. 
 Linial, Meshulam and Wallach \cite{LM}, \cite{MW} studied 
an important analogue of the classical Erd\H os--R\'enyi \cite{ER} model of random graphs in the situation of high-dimensional simplicial complexes. 
The random simplicial complexes of \cite{LM}, \cite{MW} are $d$-dimensional, have the complete $(d-1)$-skeleton and their randomness shows
only in the top dimension. Some interesting results about the topology of random 2-complexes in the Linial--Meshulam model were obtained in \cite{BHK}, \cite{CCFK}, \cite{CF1}. 


A different model of random simplicial complexes was studied by M. Kahle \cite{Kahle1} and by some other authors, see for example \cite{CFH}.
These are the clique complexes of random Erd\H os--R\'enyi graphs, i.e. here one starts with a random graph in the 
Erd\H os--R\'enyi model and declares as a simplex every subset of vertices which form a {\it clique} (a subset such that every two vertices are connected by an edge). 
Compared with the Linial - Meshulam model, the clique complex has randomness in dimension one but it influences its structure in all the higher dimensions.

In \cite{CF14} we initiated the study of a more general and more flexible model of random simplicial complexes with randomness in all dimensions. 
Here one starts with a set of $n$ vertices and retain each of them with probability $p_0$; on the next step one connects every pair of retained vertices by an edge with probability $p_1$, and then fills in every triangle in the obtained random graph with probability $p_2$, and so on. 
As the result we obtain a random simplicial complex depending on the set of probability parameters 
$$(p_0, p_1, \dots, p_r), \quad 0\le p_i\le 1.$$
Our multi-parameter random simplicial complex includes both Linial-Meshulam and random clique complexes as special cases. 

The topological and geometric properties of multi-parameter random simplicial complexes depend on the whole set of parameters and their thresholds can be understood as convex subsets and not as single numbers as in all the previously studied models. 

In this paper we develop further the multi-parameter model. Firstly, we give an intrinsic characterisation of the multi-parameter probability measure. 
Secondly, we show that in multi-parameter random simplicial complexes the links of simplexes and their intersections are also multi-parameter random simplicial complexes. Thirdly, we find conditions under which a multi-parameter random simplicial complex is connected and simply connected. 

Note that already in the case of random clique complex, the links of simplexes are two-parameter random simplicial complexes, see Example \ref{ex33}.

In \cite{CF15a} we state {\it a homological domination principle} for random simplicial complexes, claiming that the Betti number 
in one specific dimension $k=k(\p)$ (which is explicitly determined by the probability multi-parameter $\p$) significantly dominates the Betti numbers in all other dimensions. 
We also state and discuss evidence for two interesting conjectures which would imply a stronger version of the domination principle, namely that generically homology of random simplicial complex coincides with that of 
 a wedges of spheres of dimension $k=k(\p)$; moreover, for $k=k(\p)\ge 3$ a random complex collapses to a wedge of spheres of dimension 
 $k=k(\p)$.

In the following papers we shall describe the properties of fundamental groups of the multi-parameter random simplicial complexes and also their Betti numbers.

In this paper we use the following notations. 

Given a simplicial complex $Y$ and a simplex $\sigma\subset Y$ we denote by $\st_Y(\sigma)$ {\it the star} of $\sigma$ in $Y$. 
A simplex $\tau\subset Y$ belongs to the star 
$\st_Y(\sigma)$ iff the union of the sets of vertices $V(\sigma)\cup V(\tau)$ spans a simplex of $Y$. Clearly, $\st_Y(\sigma)$ is a simplicial subcomplex of $Y$. 

{\it The link of a simplex} $\sigma$ in $Y$ is 
defined as the simplicial subcomplex of $\st_Y(\sigma)$ consisting of the simplexes $\tau\subset \st_Y(\sigma)$ such that $V(\tau)\cap V(\sigma)=\emptyset$. 

This research was supported by the EPSRC.

\section{Multi-parameter random simplicial complexes}  

\subsection{Faces and external faces} \label{1.1}
   Let $\Delta_n$ denote the simplex with the vertex set $\{1, 2, \dots, n\}$. We view $\Delta_n$ as an abstract simplicial complex of dimension $n-1$. 

Given a simplicial subcomplex $Y\subset \Delta_n$, we denote by $f_i(Y)$ the number of {\it $i$-faces} of $Y$ (i.e. $i$-dimensional simplexes of $\Delta_n$ contained in $Y$). 
We shall use the symbol $F(Y)$ to denote the set of all faces of $Y$. 

\begin{definition}
An external face of a subcomplex $Y\subset \Delta_n$ is a simplex $\sigma \subset \Delta_n$ such that $\sigma \not\subset Y$ but the boundary of $\sigma$ is contained in $Y$, 
$\partial \sigma \subset Y$. 
\end{definition}
We shall denote by $E(Y)$ the set of all external faces of $Y$; the symbol $e_i(Y)$ will indicate the number of 
$i$-dimensional external faces of $Y$. 

A vertex $i\in \{1, \dots, n\}$ is an external vertex of $Y\subset \Delta_n$ iff $i\notin Y$. 
An edge $(ij)$ is an external edge of $Y$ iff $i,j\in Y$ but $(ij)\not\subset Y$. 

For $i=0$, 
we have $e_0(Y)+f_0(Y)=n$ and for $i>0$,
$$f_i(Y)+e_i(Y)\le \binom n {i+1}.$$

Note that for any simplex $\sigma\subset \left(\Delta_n- Y\right)$ which is not a simplex of $Y$ has a face $\sigma'\subset \sigma$ which is an external face of $Y$. In other words,
the complement $\Delta_n-Y$ is the union of the open stars of the external faces of $Y$, 
$$\Delta_n-Y \, =\, \bigcup_{\sigma\in E(Y)} {\rm St}(\sigma).$$ 

For two subcomplexes $Y, Y'\subset \Delta_n$, one has $Y\subset Y'$ if and only if for any external face $\sigma'$ of $Y'$ there is a face $\sigma\subset \sigma'$ which is an external face of $Y$. 

\subsection{The model}
Fix an integer $r\ge 0$ and a sequence $${\mathfrak p}=(p_0, p_1, \dots, p_r)$$ of real numbers satisfying $$0\le  p_i\le 1.$$ Denote  
$$q_i=1-p_i.$$
For a simplex $\sigma\subset \Delta_n$ we shall use the notations $p_\sigma=p_i$ and $q_\sigma=q_i$ where $i=\dim \sigma$. 

We consider the probability space ${\Omega_n^r}$ consisting of all subcomplexes 
$$Y\subset \Delta_n, \quad \mbox{with}\quad  \dim Y\le r.$$
Recall that the symbol $\Delta_n^{(r)}$ stands for the $r$-dimensional skeleton of $\Delta_n$, which is defined as the union of all simplexes of dimension $\leq r$. Thus, 
our probability space ${\Omega_n^r}$ consists of all subcomplexes $Y\subset \Delta_n^{(r)}$. 
The probability function
\begin{eqnarray}
\PP_{r,\p}: {\Omega_n^r}\to \R
\end{eqnarray} is given by the formula
\begin{eqnarray}\label{def1}
\PP_{r, \p}(Y) \, &=&\,  \prod_{\sigma\in F(Y)} p_\sigma \cdot \prod_{\sigma\in E(Y)} q_\sigma \nonumber\\  \\
&=& \, \prod_{i=0}^r p_i^{f_i(Y)}\cdot \prod_{i=0}^r q_i^{e_i(Y)}\nonumber
\end{eqnarray}
In (\ref{def1}) we use the convention $0^0=1$; in other words, if $p_i=0$ and $f_i(Y)=0$ then the corresponding factor in (\ref{def1}) equals 1; similarly if some $q_i=0$ and $e_i(Y)=0$. 

We shall show below that $\PP_{r, \p}$ is indeed a probability function, i.e. 
\begin{eqnarray}\label{sum1}
\sum_{Y\subset \Delta_n^{(r)}}\PP_{r, \p}(Y) = 1, 
\end{eqnarray}
see Corollary \ref{prob}. 

If $p_i=0$ for some $i$ then according to (\ref{def1}) we shall have $\PP_{r, \p}(Y)=0$ unless $f_i(Y)=0$, i.e. if $Y$ contains no simplexes of dimension $i$ (in this case $Y$ contains no simplexes
of dimension $\ge i$). Thus, if $p_i=0$ the probability measure $\PP_{r, \p}$ is supported on the set of subcomplexes of $\Delta_n$ of dimension $<i$. 

In the special case when one of the probability parameters satisfies $p_i=1$ one has $q_i=0$ and from formula (\ref{def1}) we see $\PP_{r, \p}(Y)=0$ unless $e_i(Y)=0$, i.e. if the subcomplex $Y\subset \Delta_n^{(r)}$ has no external faces of dimension $i$. In other words, we may say that if $p_i=1$ the measure $\PP_r$ is concentrated on the set of complexes satisfying $e_i(Y)=0$, i.e. such that any boundary of the $i$-simplex in $Y$ is filled by an $i$-simplex of $Y$.  

\begin{lemma}\label{cont}
Let $$A\subset B\subset \Delta_n^{(r)}$$ be two subcomplexes satisfying the following condition: the boundary of any external face of $B$ of dimension $\le r$ 
is contained in $A$. Then 
\begin{eqnarray}\label{twosided}
\PP_{r, \p}(A\subset Y\subset B) &=& \prod_{\sigma\in F(A)} p_\sigma \cdot \prod_{\sigma\in E(B)}q_\sigma \nonumber\\  \\
&=& \prod_{i=0}^r p_i^{f_i(A)} \cdot \prod_{i=0}^r q_i^{e_i(B)}.\nonumber
\end{eqnarray}
\end{lemma}

\begin{proof} We act by induction on $r$. 
For $r=0$, the complexes $A\subset B$ are discrete sets of vertices and the condition of the Lemma is automatically satisfied (since the boundary of any 0-face is the empty set). A subcomplex $Y\subset \Delta_n^{(0)}$ satisfying $A\subset Y\subset B$ is determined by a choice of $f_0(Y)-f_0(A)$ vertices out of $f_0(B)-f_0(A)$ vertices. Hence using formula (\ref{def1}), 
\begin{eqnarray*}
\PP_{0, \p}(A\subset Y\subset B) &=& \sum_{k=0}^{f_0(B)-f_0(A)} \binom {f_0(B)-f_0(A)} k \cdot p_0^{f_0(A)+k}q_0^{n-f_0(A)-k}\\
&=& p_0^{f_0(A)} \cdot q_0^{n-f_0(A)}\cdot \left(1+ \frac{p_0}{q_0}\right)^{f_0(B)-f_0(A)} \\
&=& p_0^{f_0(A)} \cdot q_0^{n-f_0(B)}\\
&=& p_0^{f_0(A)} \cdot q_0^{e_0(B)},
\end{eqnarray*}
as claimed. 

Now suppose that formula (\ref{twosided}) holds for $r-1$ and consider the case of $r$. Note the formula 
\begin{eqnarray}\label{ind}
\PP_{r, \p}(Y) = \PP_{r-1, \p'}(Y^{r-1}) \cdot q_r^{g_r(Y)}\cdot \left(\frac{p_r}{q_r}\right)^{f_r(Y)}
\end{eqnarray}
where $g_r(Y)=e_r(Y)+f_r(Y)$ is the number of boundaries of $r$-simplexes contained in $Y$ and $\p'=(p_0, \dots, p_{r-1})$. 
Note that the first two factors in (\ref{ind}) depend only on the skeleton $Y^{r-1}$. 

We denote by $g_r^B(Y)$ the number of $r$-simplexes of $B$ such that their boundary 
$\partial \Delta^r$ lies in $Y$. Clearly the number $g_r^B(Y)$ depends only on the skeleton $Y^{r-1}$. 
Our 
assumption that the boundary of any 
external $i$-face of $B$ is contained in $A$ for $i\le r$ implies that for any subcomplex $A\subset Y\subset B$ 
\begin{eqnarray}\label{indep}
g_r(Y)-g_r^B(Y) = e_r(B).
\end{eqnarray}

A complex $Y$ is uniquely determined by its skeleton $Y^{r-1}$ and by the set of its $r$-faces. Given the skeleton $Y^{r-1}$, the number $f_r(Y)$ is arbitrary satisfying 
$$f_r(A)\subset f_r(Y)\subset g_r^B(Y).$$
Thus using (\ref{ind}) we find that the probability  
\begin{eqnarray*}
\PP_{r,\p}(A\subset Y\subset B) = \sum_{A\subset Y\subset B} \PP_{r, \p}(Y) 
\end{eqnarray*}
equals
\begin{eqnarray*}
&\sum_{Y^{r-1}}& 
\PP_{r-1, \p'}(Y^{r-1})  \cdot q_r^{g_r(Y)}\cdot 
\sum_{k=0}^{g_r^B(Y)-f_r(A)} 
\binom {g_r^B(Y)-f_r(A)} k \cdot \left(\frac{p_r}{q_r}\right)^{f_r(A)+k} \\
&=& \sum_{Y^{r-1}} 
\PP_{r-1, \p'}(Y^{r-1})  \cdot q_r^{g_r(Y)}\cdot \left(\frac{p_r}{q_r}\right)^{f_r(A)}\cdot
\left(1+\frac{p_r}{q_r}\right)^{g_r^B(Y)-f_r(A)}\\
&=& \sum_{Y^{r-1}} 
\PP_{r-1, \p'}(Y^{r-1})  \cdot p_r^{f_r(A)} \cdot q_r^{g_r(Y)-g_r^B(Y)}\\
&=& p_r^{f_r(A)}\cdot q_r^{e_r(B))}\cdot \sum_{Y^{r-1}} \PP_{r-1, \p'}(Y^{r-1}).
\end{eqnarray*}
Here we used the equation (\ref{indep}). 
Next we may combine the obtained equality with the inductive hypothesis 
$$\PP_{r-1, \p'}({A^{r-1}\subset Y^{r-1}\subset B^{r-1}}) = \prod_{i=0}^{r-1} p_i^{f_i(A)}\cdot \prod_{i=0}^{r-1}q_i^{e_u(B)}$$ 
to obtain
(\ref{twosided}). 
\end{proof}

The assumption that any external face of $B$ is an external face of $A$ is essential in Lemma \ref{cont}; the lemma is false without this assumption. 

Taking $B=\Delta_n^{(r)}$ in Lemma \ref{cont} we obtain:

 \begin{corollary}\label{cont2c}
Let $A\subset\Delta_n^{(r)}$ be a subcomplex. Then 
\begin{eqnarray}\label{twosided1}
\PP_{r, \p}(Y\supset A) \, =\, \sum_{Y\supset A} \PP_{r,\p}(Y) \, =\,  \prod_{\sigma\in F(A)} p_\sigma 
= \prod_{i=0}^r p_i^{f_i(A)}.
\end{eqnarray}
\end{corollary}

Taking the special case $A=\emptyset$ in (\ref{twosided1}) we obtain the following Corollary confirming the fact that $\PP_{r,\p}$ is a probability function.
\begin{corollary}\label{prob}
$$\sum_{Y\subset \Delta_n^{(r)} }\PP_{r,\p}(Y) = 1. $$
\end{corollary}

\subsection{The number of vertices of $Y$}

We start with the following example. 

\begin{example}\label{emptyset}
{\rm
According to formula (\ref{def1}) the probability of the empty subcomplex $Y=\emptyset$ equals
$$\PP_{r,\p}(Y=\emptyset) = (1-p_0)^n.$$

If $p_0\to 0$ then 
$\PP_{r,\p}(Y=\emptyset) = (1-p_0)^n \sim e^{-p_0n}.$
Hence, we see that if $np_0\to 0$ then $\PP_{r,\p}(Y=\emptyset)\to 1$; we may say that in this case $Y=\emptyset$, a.a.s.

If $p_0=c/n$ then 
$$\PP_{r,\p}(Y=\emptyset) = (1-c/n)^n \to e^{-c}$$ as $n\to \infty$. 
Thus, for $p_0=c/n$ the empty subset appears with positive probability $\sim e^{-c}$, a.s.s.  

Since we intend to study non-empty large random simplicial complexes, we shall always assume that $p_0=\frac{\omega}{n}$ where $\omega$ tends to $\infty$. 
}
\end{example}

For $t\in \{0, 1, \dots, n\}$ denote by ${\Omega_{n,t}^r}$ the set of all subcomplexes $Y\subset \Delta_n^{(r)}$ with $f_0(Y)=t$. 
\begin{lemma} \label{tt} One has
\begin{eqnarray}
\sum_{Y\in \Omega_{n,t}^r} \PP_{r, \p}(Y) = \binom n t \cdot p_0^t\cdot q_0^{n-t}. \end{eqnarray} 
\end{lemma}
\begin{proof}
For a subset $A\subset \{1, 2, \dots, n\}$ with $|A|=t$ denote by $B_A\subset \Delta_n$ the $r$-dimensional skeleton of the simplex spanned by $A$. 
The pair $A\subset B_A$ satisfies the condition of Lemma \ref{cont} and applying this lemma we obtain 
$$\PP_{r, \p}(A\subset Y\subset B_A) = p_0^tq_0^{n-t}.$$
Since we have $\binom n t$ choices for $A$ the result follows. 
\end{proof}

\begin{lemma} \label{vertices} Consider a random simplicial complex $Y\in {\Omega_n^r}$ with respect to the multi-parameter probability measure $\PP_{r, \p}$ where $\p=(p_0, p_1, \dots, p_r)$.  
Assume that $p_0=\omega/n$ where $\omega \to \infty$. Then
for any $0<\epsilon<1/2$ there exists an integer $N_\epsilon$ such that for all $n>N_\epsilon$ 
with probability 
$$\ge 1- 2e^{-\frac{1}{3}\omega^{{2\epsilon}}},$$ 
 the number of vertices $f_0(Y)$ of $Y$ satisfies the inequality
\begin{eqnarray}
(1-\delta)\omega \le f_0(Y) \le (1+\delta)\omega,
\end{eqnarray}
where $$\delta=\omega^{-1/2 +\epsilon}.$$ 
\end{lemma}
\begin{proof} 
By Lemma \ref{tt}, $f_0$ is a binomial random variable with $\E(f_0)=np_0=\omega$ 
and we may apply the Chernoff bound (see \cite{JLR}, Corollary 2.3 on page 27). Let $N_\epsilon$ be such that for all $n>N_\epsilon$ we have $\delta\leq 3/2$. Then

%
%
%
$$\PP_{r, \p}(|f_0-\omega|\geq\delta \cdot\omega)<2\exp\left(-\frac{\delta^2}{3}\mathbb{E}f_0\right) = 2\exp\left(-\frac{1}{3} \omega^{2\epsilon}\right).$$
This completes the proof.
\end{proof}

\begin{example}{\rm It is easy to see that a random complex $Y \in \Omega_n^r$ is zero-dimensional a.a.s. assuming that 
\begin{eqnarray}\label{dim0}
n^2 p_0^2p_1 \to 0.
\end{eqnarray}
Indeed using Lemma \ref{cont} one finds that the expected number of edges in $Y$ is 
$$\binom n 2\cdot {p_0}^2p_1$$
and the statement follows from the first moment method. 
}
\end{example}

\subsection{Important special cases.}


The multi-parameter model we consider in this paper turns into some important well known models in special cases:

When $r=1$ and $\p=(1, p)$ we obtain the classical model of random graphs of Erd\"os and R\'enyi \cite{ER}. 

When $r=2$ and $\p= (1,1,p)$ we obtain the Linial - Meshulam model of random 2-complexes \cite{LM}. 

When $r$ is arbitrary and fixed and $\p = (1, 1, \dots, 1, p)$ we obtain the random simplicial complexes of Meshulam and Wallach \cite{MW}. 

For $r=n-1$ and $\p=(1,p, 1,1, \dots, 1)$ one obtains the clique complexes of random graphs studied in \cite{Kahle1}.

\subsection{Characterisation of the multi-parameter measure.}

In this subsection we show that the property of Corollary \ref{cont2c} is characteristic for the multi-parameter measure. 

\begin{lemma}\label{characterisation}
Let ${\mathbf P}$ be a probability measure on the set ${\Omega_n^r}$ of all $r$-dimensional subcomplexes of $\Delta_n$. Suppose that there exist real numbers $p_0, p_1, \dots, p_r\in [0,1]$ such that for any subcomplex 
$A\subset \Delta_n^{(r)}$ one has
\begin{eqnarray}
{\mathbf P}(Y\supset A) = \sum_{Y\supset A} {\mathbf P}(Y) = \prod_{i=0}^r p_i^{f_i(A)}.
\end{eqnarray}
Then ${\mathbf P}$ coincides with the measure $P_{r, \p}: {\Omega_n^r} \to \R$ given by formula (\ref{def1}) with the multi-parameter $\p=(p_0, p_1, \dots, p_r)$. 
\end{lemma}

\begin{proof}
Let $A\subset \Delta_n$ be a subcomplex. We want to show that 
$${\mathbf P}(A) = \prod_{i=0}^r p_i^{f_i(A)} \cdot \prod_{i=0}^r q_i^{e_i(A)} =P_{r,\p}(A), \quad \mbox{where}\quad q_i=1-p_i.$$
Let $E=E(A)$ denote the set of external faces of $A$. 
For each subset $S\subset E$ we denote by $A_S$ the simplicial complex $$A_S \, =\, A\cup\bigcup_{\sigma\in S}\sigma.$$
Here $S=\emptyset$ is also allowed and $A_\emptyset =A$. 
Then by our assumption concerning ${\mathbf P}$ we have $${\mathbf P}(Y\supset A_S) = {\mathbf P}(Y\supset A)\cdot \prod_{\sigma\in S}p_\sigma,$$ where $p_\sigma$ denotes $p_i$ with $i=\dim \sigma$. 

Clearly,
$$\{A\} = \{Y; Y\supset A\}- \bigcup_{\sigma\in E(A)} \{Y; Y\supset (A\cup \sigma)\}$$
and using the inclusion-exclusion formula we have (note that below $S$ runs over all subsets of $E$ including the empty set)
\begin{eqnarray*}{\mathbf P}(A) &= &
\sum_{S\subset E} (-1)^{|S|} {\mathbf P}(Y\supset A_S)\\
&=& {\mathbf P}(Y\supset A) \cdot \sum_{S\subset E}(-1)^{|S|} \prod_{\sigma\in S} p_\sigma \\ &= &
{\mathbf P}(Y\supset A) \cdot \prod_{\sigma\in E}(1-p_\sigma) \\ &=& \prod_{i=0}^r p_i^{f_i(A)}\cdot \prod_{i=0}^r q_i^{e_i(A)} \\ &=& P_{r,\p}(A).
\end{eqnarray*}

\end{proof}

\section{Links in multi-parameter random simplicial complexes.}

In this section we show that links of simplexes in a multi-parameter random complex are also multi-parameter random simplicial complexes and we find the probability multi-parameters of the links.
We also study the intersections of the links and find their probability multi-parameters.

First we consider the link of a single vertex. 


\begin{lemma}\label{link}
Let $Y\in {\Omega_n^r}$ be a multi-parameter random complex with respect to the measure $\PP_{r,\p}$, where $\p=(p_0, p_1, \dots, p_r)$.  Then the link of any vertex of $Y$ is a multi-parameter 
random complex $L\in \Omega_{n-1}^{r-1}$ with the multi-parameter $\p' = (p'_0, p'_1, \dots, p'_{r-1})$ where 
$$p'_i=p_ip_{i+1}, \quad \mbox{for} \quad i=0, 1, \dots, r-1.$$
\end{lemma}

\begin{proof} 
Let us assume that $Y$ contains the vertex $1\in \{1, \dots, n\}$. 
Then the link of $1$ in $Y$ is the union of all simplexes $(i_0, i_1, \dots, i_p)\subset \Delta_n$ such that $1<i_0<i_1< \dots<i_p\le n$ and the simplex $(1, i_0, i_1, \dots, i_p)$ is contained in $Y$. 

Let $\Delta'\subset \Delta_n$ denote the simplex spanned by $\{2, 3, \dots, n\}$. 
If $Y\subset \Delta_n$ is a subcomplex of dimension $\le r$ containing $1$, then the link of 1 in $Y$, denoted $L(Y)\subset \Delta'$, is a subcomplex of dimension $r-1$. 

We may define the following probability function on the set of all subcomplexes $L\subset \Delta'^{(r-1)}$: 
\begin{eqnarray}
{\mathbf P}(L) = p_0^{-1}\cdot \sum_{1\in Y \& L(Y)=L} \PP_{r, \p}(Y). 
\end{eqnarray}
The fact that $\mathbf P$ is a probability measure follows from Corollary \ref{cont2c} applied to the subcomplex $A=\{1\}$. 

We want to apply to the measure $\mathbf P$ the criterion of Lemma \ref{characterisation}. Hence we need to compute 
$${\mathbf P}(Z\supset L) = \sum_{Z\supset \Delta'} {\mathbf P}(Z)$$ where $Z\subset \Delta'$ runs over all subcomplexes of dimension $r-1$. Clearly we have 
\begin{eqnarray*}{\mathbf P}(Z\supset L) &=& p_0^{-1} \sum_{1\in Y \& L(Y)\supset L} \PP_{r, \p}(Y) \\ &=& p_0^{-1}\cdot \PP_{r,\p}(Y\supset CL) .\end{eqnarray*}
Here $CL\subset \Delta_n$ denotes the cone over $L$ with apex $1$. Applying Corollary \ref{cont2c} and observing that 
\begin{eqnarray*}f_i(CL) &=& f_i(L)+f_{i-1}(L), \quad i= 1, \dots, r,\\
f_0(CL)&=& f_0(L)+1,\end{eqnarray*}
we get (since $f_r(L)=0$)
$$\PP_{r,\p}(Y\supset CL) = p_0\cdot (p_0p_1)^{f_0(L)}\cdot (p_1p_2)^{f_1(L)}\cdots (p_{r-1}p_r)^{f_{r-1}(L)}.$$
Thus, 
$${\mathbf P}(Z\supset L) = \prod_{i\ge 0} \left(p_ip_{i+1}\right)^{f_i(L)}$$
and hence Lemma \ref{characterisation} implies that $L$ is a multi-parameter random complex with the multi-parameter $\p' = (p_0p_1, p_1p_2, \dots, p_{r-1}p_r).$
\end{proof}

Next we consider the general case. 

\begin{lemma}\label{linkk}
Let $Y\in {\Omega_n^r}$ be a multi-parameter random complex with respect to the measure $\PP_{r,\p}$, where $\p=(p_0, p_1, \dots, p_r)$.  
Then the link of any $k$-dimensional simplex of $Y$ (where $k<r$) is a multi-parameter 
random simplicial complex $L\in \Omega_{n-k-1, r-k-1}$ with the multi-parameter $$\p' = (p'_0, p'_1, \dots, p'_{r-k-1})$$
where
\begin{eqnarray}\label{haha}
p'_i=\prod_{j=i}^{i+k+1} p_j^{\binom {k+1}{j-i}}.
\end{eqnarray}
For example, for $k=1$ we have $$p'_i=p_ip_{i+1}^2p_{i+2},$$ and for $k=2$, $$p'_i= p_ip_{i+1}^3p_{i+2}^3p_{i+3}.$$ 
\end{lemma}

\begin{proof} Let $\sigma_0\subset \Delta_n$ be a fixed $k$-dimensional simplex; without loss of generality we may assume that $\sigma_0=(1, 2, \dots, k+1)$. 
Consider the complexes $Y\in {\Omega_n^r}$ containing $\sigma_0$; for each of these complexes let $L(Y)$ denote the link of $\sigma_0$ in $Y$. 
Clearly $L(Y)$ is a subcomplex of the simplex $\Delta'$ spanned by the vertices $k+1, \dots, n$. Since $\dim L(Y)\le r-k-1$ we may view $L(Y)$ as an element of $\Omega_{n-k-1, r-k-1}$. 
Recall that the link of $\sigma_0$ in $Y$ is the union of all simplexes $(i_0, i_1, \dots, i_p)\subset \Delta'$ such that $k+1<i_0<i_1< \dots<i_p\le n$ and the simplex 
$(1, ,2, \dots, k+1, i_0, i_1, \dots, i_p)$ is contained in $Y$. 

As in the previous Lemma, define the following probability function on the set of all subcomplexes $L\subset \Delta'^{(r-k-1)}$: 
\begin{eqnarray}\label{probinduced}
{\mathbf P}(L) = \left[\prod_{i=0}^k p_i^{\binom {k+1}{i+1}}\right]^{-1} \cdot \sum_{\sigma_0\subset Y \& L(Y)=L} \PP_{r, \p}(Y). 
\end{eqnarray}
Here $Y$ runs over all subcomplexes $Y\in {\Omega_n^r}$ containing the simplex $\sigma_0$. The first factor normalises (\ref{probinduced}) and makes it a 
 probability measure as follows from Corollary \ref{cont2c} applied to the subcomplex $A=\sigma_0$. 

We want to apply to the measure $\mathbf P$ the criterion of Lemma \ref{characterisation} and we need to compute 
$${\mathbf P}(Z\supset L) = \sum_{Z\supset \Delta'} {\mathbf P}(Z)$$ 
where $Z\subset \Delta'$ runs over all subcomplexes of dimension $r-k-1$. Clearly we have 
\begin{eqnarray*}{\mathbf P}(Z\supset L) &=& \left[\prod_{i=0}^k p_i^{\binom {k+1}{i+1}}\right]^{-1} \cdot  
\sum_{\sigma_0\subset Y \& L(Y)\supset L} \PP_{r, \p}(Y) \\  \\ &=& \left[\prod_{i=0}^k p_i^{\binom {k+1}{i+1}}\right]^{-1} \cdot  \PP_{r,\p}(Y\supset \sigma_0\ast L) .\end{eqnarray*}
Here $\sigma_0\ast L\subset \Delta_n$ denotes the join of $\sigma_0$ and $L$. 
To compute the last factor we may apply Corollary \ref{cont2c}. 
Note that
\begin{eqnarray*}f_i(\sigma_0\ast L) &=& \sum_{j=0}^{k+1} \binom {k+1} j \cdot f_{i-j}(L), \quad \mbox{for} \quad i>k\\
f_i(\sigma_0\ast L) &=& \sum_{j=0}^{i} \binom {k+1} j \cdot f_{i-j}(L)+ \binom {k+1}{i+1},\quad \mbox{for} \quad i\le k.\end{eqnarray*}
Hence, 
we get (since $f_r(L)=0$)
$$\PP_{r,\p}(Y\supset \sigma_0\ast L) =  \prod_{i=0}^k p_i^{\binom {k+1}{i+1}} \cdot \prod_{i=0}^r \prod_{j=0}^{k+1} \left[p^{\binom {k+1} j}\right]^{f_{i-j}(L)}$$
Thus, substituting in the formula above we obtain
$${\mathbf P}(Z\supset L) = \prod_{i\ge 0}^{r-k-1} \left(p'_i\right)^{f_i(L)}$$
where the numbers $p'_i$ are given by (\ref{haha}). 
This completes the proof. 
\end{proof}

\begin{example}\label{ex33} {\rm 
Let $r=n-1$ and $\p=(1, p, 1, \dots, 1)$. Hence we consider clique complexes $Y$ of Erd\"os - R\'enyi random graphs with edge probability $p$. 
The link of a vertex of $Y$ has multi-parameter
$\p'=(p, p, 1, \dots, 1)$, i.e. it has two probability parameters. 
The link of an edge of $Y$ has probability multi-parameter $(p^2, p, 1, \dots, 1)$ and the link of a 2-simplex has  
probability multi-parameter $(p^3, p, 1, \dots, 1)$. Thus, links of simplexes in clique complexes of Erd\"os - R\'enyi random graphs are also clique complexes but the underlying random graphs are of slightly more general nature as they have a vertex probability parameter $\not=1$. 
 }
\end{example}

Recall that the degree of a $k$-dimensional simplex $\sigma$ in a simplicial complex $Y$ is defined as the number of $(k+1)$-dimensional simplexes containing 
$\sigma$. Clearly the degree of $\sigma$ in $Y$ coincides with the number of vertices in the link of $\sigma$ in $Y$. Hence, applying 
Lemma \ref{tt} in combination with Lemma \ref{linkk} we obtain:

\begin{corollary}
The degree of a vertex of a random complex $Y\in \Omega_n^r$ with respect to $\PP_{r, \p}$, where $\p=(p_0, p_1, \dots, p_r)$,  
has binomial distribution ${\it Bi}(n-1, p_0p_1)$. 
In other words, probability that a vertex of $Y$ has degree $k$ equals
$$\binom {n-1} k \cdot (p_0p_1)^k \cdot (1-p_0p_1)^{n-1-k}, $$
where  $k=0,1, \dots, n-1.$
More generally, 
the degree of a $k$-dimensional simplex $\sigma$ of a random complex $Y\in \Omega_n^r$ with respect to $\PP_{r, \p}$ 
has binomial distribution ${\it Bi}(n-k-1, p)$
where $$p=\prod_{i=0}^{k+1} p_i^{\binom {k+1} i} = p_0p_1^{k+1}p_2^{\binom {k+1} 2} \cdots p_k^{k+1} p_{k+1}.$$
In other words, probability that a $k$-simplex of $Y$ has degree $k$ equals
$$\binom {n-k-1} k \cdot p^k \cdot (1-p)^{n-1-k}, $$
where $k=0,1, \dots, n-k-1.$
\end{corollary}

The following Corollary will be used later in this paper.

\begin{corollary} \label{degreezero} Assume that $p_0=\omega/n$ with $\omega\to \infty$ and 
\begin{eqnarray}
p_1^2\cdot p_2 \, \ge\,   \frac{2\log \omega + c}{\omega},
\end{eqnarray}
where $c$ is a constant. 
Then there exists $N$ (depending on the sequence $\omega$ and on $c$) 
such that for all $n>N$ the probability that a random complex $Y\in \Omega_n^r$ with respect to the measure $\PP_{r, \p}$ has an edge of degree zero 
is less than 
$$p_1\cdot e^{2-c}.$$
\end{corollary} 

\begin{proof} An edge of $Y$ has degree zero if and only if its link in $Y$ is the empty set $\emptyset$. Since the link of an edge has the 
multi-parameter $(p'_0, p'_1, \dots, p'_{r-2})$ where $$p'_i = p_ip_{i+1}^2p_{i+2}$$ (see Lemma \ref{linkk}) we may apply the result of Example \ref{emptyset} to obtain that the probability that a given edge of $Y$ has degree zero equals
$$(1-p_0p_1^2p_2)^{n-2}.$$
Hence the expectation of the number of the degree zero edges in $Y$ equals to
\begin{eqnarray*}
\binom n 2 \cdot p_0^2\cdot p_1\cdot \left( 1-p_0p_1^2p_2  \right)^{n-2} &\le& \frac{1}{2}\cdot \omega^2\cdot p_1\cdot \exp\left(-p_0p_1^2p_2(n-2)\right)\\
&\le & \frac{1}{2}\cdot p_1 \cdot \exp\left(2\log \omega - (2\log \omega +c)\cdot \frac{n-2}{n}\right)\\
&=& \frac{1}{2}\cdot p_1\cdot \left\{\exp\left(\frac{4\log \omega}{n}\right)\cdot e^{\frac{2c}{n}}\right\}\cdot e^{-c}. 
\end{eqnarray*}
The first factor in the figure brackets tends to $1$ and hence it is less than $2$ for large $n$. The second factor in the figure brackets is less than $e^2$ 
(since our assumptions imply $c<\omega\le n$). This complex the proof. 
\end{proof}

Next we consider the intersections of links of several vertices. 

\begin{lemma} \label{linksintersection}
Let $k<n$ be fixed integers and
let $Y\in {\Omega_n^r}$ be a random $r$-dimensional simplicial complex with probability multi-parameter $\p=(p_0, \dots, p_r)$. 
Consider the intersection $L(Y)$ of links of $k$ distinct vertices of $Y$. Then $L(Y)\in \Omega_{n-k}^{r-1}$ is a random simplicial complex with respect to the multi-parameter 
$\p'=(p'_0, \dots, p'_{r-1})$ where
\begin{eqnarray}\label{klinks}
p'_i =p_ip_{i+1}^k, \quad\mbox{for}\quad  i=0, \dots, r-1.\end{eqnarray}
\end{lemma}
\begin{proof}
Consider the set $\Omega'$ of simplicial complexes $Y\in {\Omega_n^r}$ containing the vertices $\{1, \dots, k\}$; the function $Y\mapsto p_0^{-k}\PP_{r, n}(Y)$ is a probability measure on $\Omega'$
(by Corollary \ref{cont2c}). 
For $Y\in \Omega'$ let $L_i(Y)$ denote the link of the vertex $i$ in $Y$ where $i=1, \dots, k$. 
Let $\Delta'$ denote the simplex spanned by the remaining vertices $k+1, k+2, \dots, n$. The intersection $L(Y)=L_1(Y)\cap \dots\cap L_k(Y)$ is a subcomplex of $\Delta'$ of dimension $\le r-1$. We obtain a map 
$$\Omega' \to \Omega_{n-k}^{r-1}, \quad Y\mapsto L(Y)$$
and we wish to describe the pushforward measure $\mathbf P$ on $\Omega_{n-k}^{r-1}$ which (by the definition) is given by ${\mathbf P}(Z)=p_0^{-k}\cdot \sum_{L(Y)=Z} \PP_{r, \p}(Y).$

Given a subcomplex $L\subset \Delta'$, consider the quantity 
$${\mathbf P}(Z\supset L) = \sum_{L\subset Z\subset \Delta'} {\mathbf P}(Z);$$
in the sum $Z$ runs over all subcomplexes of $\Delta'$ satisfying $L\subset Z\subset \Delta'$, $\dim Z\le r-1$. By the construction of ${\mathbf P}$, we have 
$${\mathbf P}(Z\supset L) = p_0^{k} \cdot \sum_{L(Y)\supset L} \PP_{r, \p}(Y);$$
here $Y$ runs over all subcomplexes $Y\subset \Delta_n^{(r)}$ containing the vertices $1, \dots, k$ and such that $L(Y)\supset L$. These last two conditions can be expressed by saying that 
$Y$ contains the join 
$$J=\{1, 2, \dots, k\}\ast L$$ 
as a subcomplex. Note that $J$ is the union of $k$ cones over $L$ with vertices at the points $1, \dots, k$. One has 
\begin{eqnarray*}
f_i(J) &=& f_i(L)+kf_{i-1}(L), \quad \mbox{for}\quad i>0, \\
f_0(J) &=& f_0(L) +k.
\end{eqnarray*}
Thus, using Corollary \ref{cont2c} we obtain
\begin{eqnarray*}
{\mathbf P}(Z\supset L) &=& p_0^{-k} \cdot \sum_{J\subset Y} \PP_{r, \p}(Y) \\
&=& p_0^{-k}\cdot \prod_{i=0}^r p_i^{f_i(J)}\\
&=& (p_0p_1^k)^{f_0(L)}\cdots (p_{r-1}p_r^k)^{f_{r-1}(L)}.
\end{eqnarray*}
Finally we apply Lemma \ref{characterisation} which implies that $\mathbf P$ is a multi-parameter probability measure on $\Omega_{n-k, r-1}$ with respect to the multi-parameter (\ref{klinks}). 

\end{proof}

\section{Intersections of random complexes}

In this section we show that intersection of multi-parameter random simplicial complexes is also a multi-parameter random simplicial complex with respect to the product of multi-parameters. 

Let $Y, Y'\in {\Omega_n^r}$ be two simplicial subcomplexes of $\Delta_n^{(r)}$. Suppose that both $Y, Y'$ are random and that their probability measures are $\PP_{r, \p}$ and $\PP_{r, \p'}$, correspondingly, see (\ref{def1}). Here $\p = (p_0, \dots, p_r)$ and $\p'=(p'_0, \dots, p'_r)$ are the corresponding probability multi-parameters. 
The intersection $$Z=Y\cap Y'\, \in {\Omega_n^r}$$ 
appears with probability 
\begin{eqnarray}\label{intersection}
{\mathbf P}(Z) = \sum_{Y\cap Y'=Z} {\PP}_{r, \p}(Y)\cdot {\PP}_{r, \p'}(Y').
\end{eqnarray}
This measure $\mathbf P$ is the pushforward of the product measure $\PP_{n,\p}\times \PP_{n, \p'}$ under the map 
$${\Omega_n^r}\times {\Omega_n^r} \to {\Omega_n^r}, \quad (Y, Y')\mapsto Y\cap Y'.$$

\begin{lemma} For $Z\in {\Omega_n^r}$ one has 
$${\mathbf P}(Z) = \PP_{r, \p\p'}(Z),$$
where $\p\p'=(p_0p'_0, p_1p'_1, \dots, p_rp'_r)$. 
\end{lemma}
\begin{proof}
Let ${\mathbf P}$ be the probability measure on ${\Omega_n^r}$ given by (\ref{intersection}). To apply Lemma \ref{characterisation} we compute
\begin{eqnarray*}
{\mathbf P}(Z\supset A) &=& \sum_{Z\supset A} {\mathbf P}(Z) \\
&=& \sum_{Y\cap Y'\supset A} {\PP}_{r, \p}(Y)\cdot {\PP}_{r, \p'}(Y') \\
&=& \left[ \sum_{Y\supset A} \PP_{r, \p}(Y) \right]\cdot \left[ \sum_{Y'\supset A} \PP_{r, \p'}(Y')  \right]\\
&=& \prod_{i\ge 0} p_i^{f_i(A)} \cdot \prod_{i\ge 0} {p'_i}^{f_i(A)} \\
&=& \prod_{i\ge 0} (p_ip'_i)^{f_i(A)} .
\end{eqnarray*}
Now, Lemma \ref{characterisation} gives ${\mathbf P}(Z) = \PP_{r, \p\p'}(Z)$ for any $Z\in {\Omega_n^r}$. 

\end{proof}

Lemma \ref{intersection} suggests a way how a random simplicial complex $Y\in {\Omega_n^r}$ with general probability multi-parameter $\p=(p_0, p_1, \dots, p_r)$ can be generated. 
Consider the probability multi-parameter $$\p_i=(1, \dots, 1, p_i, 1, \dots, 1)$$ 
where $p_i$ occurs on the place with index $i$, where $i=0, 1, \dots, r$. 
Generate a random complex $Y_i\in {\Omega_n^r}$ with respect to the measure $\p_i$. Then the intersection $$Y=Y_0\cap Y_1\cap \dots \cap Y_r$$ is a random complex 
with respect to the original measure 
$\PP_{n, \p}$. Hence, $\PP_{n,\p}$ is the pushforward of the product measure $\PP_{r, \p_0}\times \PP_{r, \p_1}\times\cdots\times \PP_{r, \p_r}$ with respect to the map
$${\Omega_n^r}\times \cdots \times {\Omega_n^r} \to {\Omega_n^r}, \quad (Y_0, \dots, Y_r)\mapsto \bigcap_{i=0}^r Y_i.$$
Note that a random complex with respect to $\PP_{r, \p_i}$ has the following structure: 
(1) we start with the full $(i-1)$-dimensional skeleton $\Delta_n^{(i-1)}$ and (2) add $i$-dimensional simplexes at random, independently of each other, with probability $p_i$ (as in the Linial - Meshulam model), and then (3) we subsequently add all the external $j$-dimensional faces to the complex obtained on the previous step for $j=i+1, j+2, \dots, r$. 

\begin{example}{\rm
Consider the following construction. Start with a multi-para\-me\-ter random simplicial complex $Y$ with the multi-parameter $\p=(p_0, \dots, p_r)$. 
Let $\Delta'\subset \Delta_n$ denote the simplex 
spanned by the vertices $2, 3, \dots, n$. We claim that the intersection 
$$Y\cap \Delta'\in \Omega_{n-1}^{r}$$ is a multi-parameter random simplicial complex with the same multi-parameter $\p$. Indeed, 
the pushforward of the measure $\PP_{r, \p}$ under the map 
$$\Omega_{n}^{r} \to \Omega_{n-1}^{r}, \quad Y\mapsto Y\cap \Delta'$$
is $${\mathbf P}(Z) = \sum_{Y\cap \Delta'=Z} \PP_{r, \p}(Y).$$
For a subcomplex $A\subset \Delta'$ we have (using Corollary \ref{cont2c})
$${\mathbf P}(Z\supset A) = \sum_{Y\supset A} \PP_{r, \p}(Y) = \prod_{i\ge 0} p_i^{f_i(A)}.$$ 
Now the result follows from Lemma \ref{characterisation}. 

}\end{example}

\section{Isolated subcomplexes}\label{secisolated}

We shall say that a simplicial subcomplex $S\subset Y$ is {\it isolated} if 
no edge of $Y$ connects a vertex of $S$ with a vertex of $Y$ which is not in $S$. In other words, $S\subset Y$ is isolated if it is a union of several connected components of $Y$.

\begin{lemma}\label{lisolated}
Given a subcomplex $S\subset \Delta_n^{(r)}$, and let $Y\in {\Omega_n^r}$ be a random simplicial complex with respect to a multi-parameter $\p=(p_0, \dots, p_r)$. The probability that 
$Y$ contains $S$ as an isolated subcomplex equals 
\begin{eqnarray}\label{isolated}
\left[q_0 +p_0\cdot q_1^{f_0(S)}\right]^{n-f_0(S)}\cdot \prod_{i=0}^r p_i^{f_i(S)}.
\end{eqnarray}
\end{lemma}
\begin{proof}
Let $K$ be a subset of $\{1, \dots, n\}-V(S)$ where $V(S)$ denotes the set of vertices of $S$. 
Denote $A_K=S\cup K$ and $B_K=\Delta_S \cup \Delta_K$, where 
$\Delta_S$ and $\Delta_K$ are the simplexes spanned by $V(S)$ and $K$ respectively. 
The pair $A_K\subset B_K$ satisfies the condition of Lemma \ref{cont}. Indeed, the external faces of $B_K$ are the vertices $i\notin V(S)\cup K$ and the edges connecting the elements of 
$K$ and the vertices of $S$; these external faces of $B_K$ are also external faces of $A_K$. 

Denoting $|K|=k$ we have 
\begin{eqnarray*}f_0(A_K) &=& f_0(S) +k,\\
f_i(A_K)&=&f_i(S), \quad i\ge 1,\\
e_0(B_K)&=&n-f_0(S) -k,\\
e_1(B_K)&=& k\cdot f_0(S),\\
e_i(B_K) &=& 0, \quad \mbox{for}\quad i\ge 2.\end{eqnarray*}
Therefore, applying Lemma \ref{cont} we find
\begin{eqnarray*}
\PP_{r, \p}(A_K\subset Y\subset B_K) &=& \prod_{i=0}^r p_i^{f_i(A_K)}\cdot \prod_{i=0}^r q_i^{e_i(B_K)}\\
&=& p_0^k \cdot q_0^{n-f_0(S)-k}\cdot q_1^{kf_0(S)}\cdot \prod_{i=0}^r p_i^{f_i(S)}\\
&=& \left[\frac{p_0q_1^{f_0(S)}}{q_0}\right]^k \cdot q_0^{n-f_0(S)}\cdot \prod_{i=0}^r p_i^{f_i(S)}.
\end{eqnarray*}

Clearly, the probability that $Y$ contains $S$ as an isolated subcomplex equals the sum  
$$\sum_K \PP_{r, \p}(A_K\subset Y\subset B_K) $$ where $K$ runs over all subsets of $\{1, \dots, n\}-V(S)$. 
Hence we obtain that the desired probability equals
\begin{eqnarray*}
&&\sum_K \PP_{r, \p}(A_K\subset Y\subset B_K) \\ &=& q_0^{n-f_0(S)}\cdot \prod_{i=0}^r p_i^{f_i(S)}\cdot \sum_{k=0}^{n-f_0(S)} \binom n k \cdot \left[\frac{p_0q_1^{f_0(S)}}{q_0}\right]^k\\
&=& q_0^{n-f_0(S)}\cdot \prod_{i=0}^r p_i^{f_i(S)}\cdot \left[ 1+ \frac{p_0q_1^{f_0(S)}}{q_0}\right]^{n-f_0(S)}
\end{eqnarray*}
which is equivalent to formula (\ref{isolated}). 
\end{proof}

\begin{example}\label{vertex}{\rm 
Consider the special case when the complex $S\subset \Delta_n^{(r)}$ is a single point, $S=\{i\}$.
We obtain that the probability that $Y$ contains the vertex $\{i\}$ as an isolated point equals $$p_0(1-p_0p_1)^{n-1}.$$
}\end{example}

%
\begin{example}\label{tree}{\rm 
Let $S_v\subset \Delta_n^{(r)}$ be a tree with $v$ vertices. 
Then the probability that $Y$ contains $S_v$ as an isolated subcomplex equals 
\begin{eqnarray}
\left[q_0+p_0\cdot q_1^v\right]^{n-v}\cdot p_0^v\cdot p_1^{v-1}.
\end{eqnarray}
}\end{example}

We shall use the results of Examples \ref{vertex} and \ref{tree} to describe the range of the probability multi-parameter for which the random complex contains an isolated vertex, a.a.s.

\begin{lemma}\label{lm54}
Let $Y\in \Omega_n^r$ be a random complex  with respect to the probability measure $\PP_{r, \p}$, where $\p=(p_0, p_1, \dots, p_r)$. As above, we shall assume that 
$p_0=\omega/n$, where $\omega\to \infty$. Then: \newline
(A) If  
\begin{eqnarray}\label{minus} p_1=\frac{\log \omega - \omega_1}{\omega},\end{eqnarray}
for a sequence $\omega_1\to \infty$ then a random complex $Y\in \Omega_n^r$ contains an isolated vertex, a.a.s. In particular, under this condition a random complex $Y$ is disconnected, a.a.s. 
\newline
(B) If 
\begin{eqnarray}\label{plus} p_1=\frac{\log \omega + \omega_1}{\omega},\quad \omega_1\to \infty,\end{eqnarray}
then a random complex $Y\in\Omega_n^r$ contains no isolated vertexes, a.a.s.
\end{lemma}
We shall see below in \S \ref{seccon} that under condition (\ref{plus}) a random complex $Y$ is connected, a.a.s. 
\begin{proof}
For $i\in \{1, \dots, n\}$ let $X_i: \Omega_{n}^r \to \R$ denote the random variable $X_i(Y)=1$ if $Y$ contains $i$ as an isolated vertex, otherwise $X_i(Y)=0$. 
The sum $X=\sum_{i=1}^n X_i$ counts the number of isolated vertexes in random simplicial complexes. 
By Example \ref{vertex} we have 
$$\E(X)=np_0(1-p_0p_1)^{n-1}.$$

First, we shall assume (\ref{minus}). Then 
$$\E(X)= \omega\left(1- \frac{\log\omega -\omega_1}{n}\right)^{n-1}$$
and denoting $x= \frac{1}{n}(\log \omega -\omega_1)$ and using the power series expansion for $\log(1-x)$ we obtain
\begin{eqnarray*}
\log \E(X) &=& \log \omega -(n-1)\left[x +\frac{1}{2}x^2 +\frac{1}{3}x^3 +\dots\right]\\ \\
&=& \frac{1}{n}\log \omega +\frac{n-1}{n}\cdot  \omega_1 -(n-1)\cdot x^2\cdot [\frac{1}{2}+\frac{1}{3}x + \frac{1}{4} x^2 + \dots]\\ \\
&\ge &  \frac{n-1}{n}\cdot  \omega_1 \, -\, 1. 
\end{eqnarray*}
Here we used that $x=p_0p_1\to 0$ and
$$nx^2 \le n\cdot \frac{(\log \omega)^2}{n^2} = \frac{(\log \omega)^2}{n}\le \frac{(\log n)^2}{n} \to 0.$$
Therefore, the expectation $\E(X)$ tends to infinity. 

To show that $X>0$ under the assumption (\ref{minus}) we shall apply the Chebyshev inequality in the form 
\begin{eqnarray}\label{cheb1}
\PP_{r, \p}(X=0) \le \frac{\E(X^2)}{\E(X)^2} -1. 
\end{eqnarray}
Hence statement (A) of the lemma follows once we know that the ratio $\frac{\E(X^2)}{\E(X)^2}$ tends to $1$. 
Clearly $\E(X^2) = \sum_{i,j}\E(X_iX_j)$ and for $i\not=j$ the number $\E(X_iX_j)$ is the probability that 
 $i$ and $j$ are isolated vertices of $Y$. 
Obviously, this probability equals the difference $a-b$ where $a$ is the probability that $Y$ contains the complex $S=\{i,j\}$ as an isolated subcomplex and $b$ is the probability that 
$Y$ contains the edge $(ij)$ as an isolated subcomplex. 
Applying Lemma \ref{lisolated} one obtains that $a= \left[q_0+p_0q_1^2\right]^{n-2}p_0^2$ while $b=\left[q_0+p_0q_1^2\right]^{n-2}p_0^2p_1$ and hence 
 for $i\not=j$, 
$$\E(X_iX_j) = \left[q_0+p_0q_1^2\right]^{n-2}\cdot p_0^2\cdot q_1.$$ 
We obtain 
$$\mathbb{E}(X^2)=\mathbb{E}(X)+(n^2-n)\cdot (q_0+p_0q_1^2)^{n-2}\cdot p_0^2\cdot q_1. $$
Hence
\begin{eqnarray*}
\frac{\E(X^2)}{\E(X)^2}= \E(X)^{-1} + \left(1-\frac{1}{n}\right) \cdot \left[1+ \frac{p_0q_0p_1^2}{(1-p_0p_1)^2}\right]^{n-2}\cdot \frac{q_1}{(1-p_0p_1)^2}.
\end{eqnarray*}
The first summand $\E(X)^{-1}$ tends to $0$ (as shown above). Denoting $$y= \frac{p_0q_0p_1^2}{(1-p_0p_1)^2}$$ we observe that 
$$ny=\frac{(\log \omega - \omega_1)^2}{\omega}\cdot \frac{q_0}{(1-p_0p_1)^2} $$
tends to $0$ as $n\to \infty$. Hence 
\begin{eqnarray}\label{power}
1\le \left[1+\frac{p_0q_0p_1^2}{(1-p_0p_1)^2}\right]^{n-2} \le  \sum_{k=0}^{n-2} [(n-2)y]^k \le \frac{1}{1-(n-2)y}.
\end{eqnarray} 
and both sides of this inequality 
tend to $1$. Hence we conclude that the ratio 
$$\frac{\E(X^2)}{\E(X)^2}$$
tends to $1$ as $n\to \infty$ and (\ref{cheb1}) implies that a random complex $Y\in \Omega_n^r$ contains an isolated point with probability $\to 1$ as $n\to \infty$. 

Next we prove statement (B) under the assumption (\ref{plus}). 
We use the first moment method and show that the expectation $\E(X)$ tends to zero if (\ref{plus}) holds. 
As above, we have
\begin{eqnarray*}
\E(X) &=& np_0\left(1-p_0p_1\right)^{n-1} \\
&=& \omega \cdot \left(1-\frac{\log \omega + \omega_1}{n}\right)^{n-1} \\
&<&  \omega \cdot e^{-\frac{\log \omega +\omega_1}{n}\cdot (n-1)}\\
&=& \omega^{\frac{1}{n}}\cdot e^{-\frac{n-1}{n} \cdot \omega_1}. 
\end{eqnarray*}
The logarithm of the first factor $\frac{1}{n}\log \omega \le \frac{1}{n}\log n$ tends to zero and hence the first factor $\omega^{\frac{1}{n}}$ is bounded. 
Clearly,
the second factor tends to $0$ as $n\to \infty$. 
Thus, by the first moment method, a random complex $Y\in \Omega_n^r$ has an isolated vertex with probability tending to $0$ with $n$. 
\end{proof}

\section{Connectivity of random complexes}\label{seccon}

In this section we find the range (threshold) of connectivity of a multi-parameter random simplicial complex $Y\in \Omega_n^r$ 
with respect to the probability measure $\PP_{r, \p}$ where 
$$\p=(p_0, p_1, \dots, p_r)$$ is the multi-parameter. 
Everywhere in this section we shall assume 
that 
\begin{eqnarray}\label{everywhere}
p_0=\frac{\omega}{n}, \quad \mbox{where}\quad \omega \to \infty.
\end{eqnarray}
This is to ensure that the number of vertices of $Y$ tends to $\infty$. The connectivity depends only on the 1-skeleton and hence only the parameters $p_0$ and $p_1$ are relevant. 
Our treatment in this section is similar to the classical analysis of the connectivity of random graphs in the Erd\H{o}s--R\'{e}nyi model with an extra difficulty which arises due to the number of vertices being random. In the following section we apply Theorem \ref{propcon} to establish the region of simple connectivity of multi-parameter random simplicial complexes; this region depends on combination of the parameters $p_0, p_1, p_2$. 

The following is the main result of this section. 

\begin{theorem}\label{propcon} Consider a random simplicial complex $Y\in\Omega_{n}^{r}$ (where $r\ge 1$) with respect to a multi-parameter probability measure $\PP_{r, \p}$ satisfying (\ref{everywhere}). 
Assume that
\begin{eqnarray}\label{cc} p_1\, \ge\,  \frac{k\log \omega+c}{\omega}\end{eqnarray} 
for 
an integer $k\ge 1$ and a constant $c>0$.  
Then there exists a constant $N>0$ (depending on the sequence $\omega$ and on $c$) such that for all $n>N$ the complex $Y$ is connected with probability greater than 
\begin{eqnarray}\label{probab}
1-Ce^{-c}\omega^{1-k},
\end{eqnarray}
where $C$ is a universal constant. 
\end{theorem}

\begin{corollary}\label{corconn} If additionally to (\ref{everywhere}) one has
 $$p_1=\frac{\log\omega+\omega_1}{\omega},$$ for a sequence $\omega_1\to \infty$
then a random complex $Y\in \Omega_{n}^{r}$ with respect to $\PP_{r, \p}$  is connected, a.a.s..
\end{corollary}

Corollary \ref{corconn} complements the statement of part (A) of Lemma \ref{lm54} saying that a random complex $Y\in \Omega_n^r$ is disconnected if 
$p_1=\frac{\log\omega-\omega_1}{\omega}.$

Corollary \ref{corconn} follows from Theorem \ref{propcon} in an obvious way. 

\begin{example}\label{example63}{\rm 
Assume that $p_0=n^{-\alpha_0}$ and $p_1=n^{-\alpha_1}$
 where $\alpha_0, \alpha_1\ge 0$ are constants. In this special case Corollary 
\ref{corconn} implies that a random simplicial complex $Y\in \Omega_n^r$ is {\it connected} for 
\begin{eqnarray}\label{condomain}
\alpha_0+\alpha_1<1. 
\end{eqnarray}
Note that part (A) of Lemma \ref{lm54} implies that $Y$ is {\it disconnected} if 
\begin{eqnarray}\label{discondomain}
\alpha_0+\alpha_1>1. 
\end{eqnarray}

}
\end{example}

\begin{proof}[Proof of Theorem \ref{propcon}.]

For $v\ge 1$ let $E_v\subset \Omega_n^r$ denote the set of disconnected simplicial complexes $Y\subset \Delta_n$ such that 
$$v=\min_{j\in J} f_0(Y_j),$$
where $$Y=\sqcup_{j\in J} Y_j$$ is the decomposition of $Y$ into the connected components. 
In other words, $v$ is the smallest number of vertices contained in a single connected component of $Y\in E_v$. For $t=0, 1, \dots, n$ we denote by $E_{v,t}$ the intersection 
$$E_{v,t} = E_v \cap \Omega_{n,t}^r$$
where $\Omega_{n,t}^r$ is the set of all complexes $Y\in \Omega_{n,r}$ with $f_0(Y)=t$. 
Clearly, a complex $Y\in \Omega_{n,t}^r$ is disconnected if and only if $Y\in E_{v,t}$ for some $1\le v\le t/2$. 
By Lemma \ref{vertices}, for any fixed $\epsilon \in (0, 1/2)$,
\begin{eqnarray}\label{25}
\sum_{|t-\omega|> \delta\omega} \PP_{r, \p}(\Omega_{n,t}^r) \le 2\exp(-\frac{1}{3}\omega^{2\epsilon}). 
\end{eqnarray}
where 
\begin{eqnarray}\label{delta}\delta= \omega^{-\frac{1}{2}+\epsilon}.\end{eqnarray} 
(One may assume everywhere below that $\epsilon = 1/4$). 
Our goal is to estimate above the sum
\begin{eqnarray}\label{sum11}
\sum_{|t-\omega|\le\delta \omega}\,   \sum_{v\ge 1}^{t/2} \PP_{r, \p}(E_{v,t})
\end{eqnarray}
%
since (using (\ref{25})),
\begin{eqnarray}\label{sum2}
\PP_{r,\p}(Y; b_0(Y)>1) \le  \sum_{|t-\omega|<\delta \omega}\,   \sum_{v\ge 1}^{t/2} \PP_{r, \p}(E_{v,t})  +2\exp(-\frac{1}{3} \omega^{2\epsilon}).
\end{eqnarray}
The left hand side of (\ref{sum2}) is the probability that $Y$ is disconnected (i.e. its zero-dimensional Betti number $b_0(Y)$ is greater than $1$.) 
Hence an upper bound for the sum (\ref{sum11}) 
will give an upper bound on the probability that $Y$ is disconnected.

For a tree $T\subset \Delta_n$ on $v$ vertices and for a subset $K\subset \{1, \dots, n\} -F_0(T)$ of cardinality $t-v$,
denote
$$A_{T,K} = T\cup K, \quad B_{T,K}= \Delta_S \cup \Delta_K,$$
where $S=V(T)$ is the set of vertices of $T$ and $\Delta_S$ and $\Delta_T$ denote the simplexes spanned by $S$ and $T$ correspondingly. 
The pair of subcomplexes $A_{T,K}\subset B_{T,K}$ satisfies the condition of Lemma \ref{cont}. 
Let $P_{T,K}$ denote the probability
$$P_{T,K} \, = \, \PP_r(A_{T,K}\subset Y\subset B_{T,K})= p_0^{t} p_1^{v-1}q_0^{n-t}q_1^{v(t-v)},$$
where we have used Lemma \ref{cont}. 
Any complex $Y\in E_{v,t}$ satisfies $A_{T,K}\subset Y\subset B_{T, K}$ for a tree $T$ on $1\le v\le t/2$ vertices and for a unique subset $K$ of cardinality $t-v$. Hence, taking into account the Cayley formula for the number of trees on $v$ vertices we obtain
\begin{eqnarray*}
\PP_{r, \p}(E_{v,t}) &\le& \binom n v \cdot \binom {n-v} {t-v} \cdot v^{v-2}  \cdot P_{T,K} \\
&=&   \binom n v \cdot \binom {n-v} {t-v} \cdot v^{v-2}  \cdot   p_0^{t}\cdot p_1^{v-1}\cdot q_0^{n-t}\cdot q_1^{v(t-v)}\\
&=& \binom n t \cdot \binom t v \cdot v^{v-2}  \cdot   p_0^{t}\cdot p_1^{v-1}\cdot q_0^{n-t}\cdot q_1^{v(t-v)}.
\end{eqnarray*}
Therefore we have
\begin{eqnarray*}
 \sum_{t=(1-\delta)\omega}^{(1+\delta)\omega} \sum_{v=1}^{t/2} \PP_{r, \p}(E_{v,t}) &\le&  \sum_{t=(1-\delta)\omega}^{(1+\delta)\omega} \sum_{v=1}^{t/2}
\binom n t \cdot \binom t v \cdot v^{v-2}  \cdot   p_0^{t}\cdot p_1^{v-1}\cdot q_0^{n-t}\cdot q_1^{v(t-v)}\\
&=&  \sum_{t=(1-\delta)\omega}^{(1+\delta)\omega} \binom n t \cdot p_0^t\cdot q_0^{n-t} \cdot \sum_{v=1}^{t/2} \binom t v \cdot v^{v-2}\cdot p_1^{v-1}\cdot q_1^{v(t-v)}.
\end{eqnarray*}

Our plan is to show that there exists $N>0$ such that 
for the values of $t$ lying in the interval $[(1-\delta)\omega, (1+\delta)\omega]$ and for all $n>N$
the internal sum 
\begin{eqnarray}\label{internal}\sum_{v=1}^{t/2} \binom t v \cdot v^{v-2}\cdot p_1^{v-1}\cdot q_1^{v(t-v)}\end{eqnarray}
can be estimated above by $C\omega^{1-k}e^{-c}$ where $C$ is a universal constant. 
Then we will have 
\begin{eqnarray*} \sum_{t=(1-\delta)\omega}^{(1+\delta)\omega} \sum_{v=1}^{t/2} \PP_{r, \p}(E_{v,t})&\le& C\omega^{1-k}e^{-c}
 \sum_{t=(1-\delta)\omega}^{(1+\delta)\omega} \binom n t \cdot p_0^t\cdot q_0^{n-t}\\ &\le& C\omega^{1-k}e^{-c}
\end{eqnarray*}
which together with (\ref{sum2}) will complete the proof of Theorem \ref{propcon}. Note that the summand $2\exp(-\frac{1}{3} \omega^{2\epsilon})$ which appears in (\ref{internal}) is less than $\omega^{1-k}e^{-c}$ for $n$ large enough. 

For the term with $v=1$ we have
\begin{eqnarray*}
tq_1^{t-1} &=& t(1-p_1)^{t-1} \\ &\le& (1+\delta)\omega \cdot \exp(-p_1(t-1))\\
&=&  (1+\delta)\omega \cdot \exp(-p_1t)\cdot \exp(p_1)\\
&\le & (1+\delta)e\cdot \omega \cdot \exp(-\frac{k\log \omega +c}{\omega}\cdot (1-\delta)\omega)\\
&=& (1+\delta) e\cdot \omega^{1-k +k\delta}\cdot e^{-c(1-\delta)}\\
&=& \left\{(1+\delta)e\omega^{k\delta}e^{c\delta}\right\}\cdot \omega^{1-k}e^{-c}\\
&\le& 2e \cdot \omega^{1-k}e^{-c}
\end{eqnarray*}
for $n$ large enough. Here we used the fact that the expression in the figure brackets tends to $e$ for $n\to \infty$. Note that the factor $\omega^{k\delta}$ tends to $1$ as follows from the definition of $\delta$, see (\ref{delta}). 

Next consider the term with $v=2$: 
\begin{eqnarray*}
\binom t 2 \cdot p_1\cdot (1-p_1)^{2(t-2)} &\le& t^2\exp(-p_1(t-2)\cdot 2) \\ &=& t^2\cdot \exp(-2p_1t)\cdot \exp{4p_1}\\
&\le& e^4t^2\exp\left(-\frac{k\log\omega+c}{\omega}\cdot 2(1-\delta)\omega\right)\\
&\le& \left\{e^4(1+\delta)^2\omega^{2k\delta}e^{2\delta c} \right\}\cdot\omega^{2-2k} \cdot e^{-2c}\\
&\le& 2e^4 \omega^{1-k}e^{-c}
\end{eqnarray*}
for $n$ large enough. We used the fact that the expression in the figure brackets tends to $e^4$ for $n\to \infty$. 

Consider now a term with $v\ge 3$. Using the Stirling's formula we have 
\begin{eqnarray}\label{three}\binom t v v^{v-2} \le \frac{t^vv^{v-2}}{v!}  \le \frac{(et)^v}{\sqrt{2\pi}v^{5/2}} \le (3t)^v \le (3(1+\delta)\omega)^v \le (6\omega)^v.
\end{eqnarray}
The function $x\mapsto x^{v-1}(1-x)^{v(t-v)}$ is decreasing for 
$\frac{v-1}{v-1+v(t-v)} <x<1$. Hence, observing that for $n$ large enough 
$$\frac{v-1}{v-1+v(t-v)}\le \frac{1}{t-v+1} \le 2t^{-1} 
\le \frac{2}{(1-\delta)\omega}\le \frac{k\log \omega+ c}{\omega} \le p_1 \le 1$$
we obtain 
\begin{eqnarray*}
p_1^{v-1}q_1^{v(t-v)} &\le & \left[\frac{k\log \omega +c}{\omega}\right]^{v-1}\cdot \left(1-\frac{k\log\omega+c}{\omega}\right)^{v(t-v)} \\
&\le&  \left[\frac{k\log \omega +c}{\omega}\right]^{v-1}\cdot \exp\left(-\frac{k\log \omega +c}{\omega}\cdot (t-v)\right)^v\\
&\le & \left[\frac{k\log \omega +c}{\omega}\right]^{v-1}\cdot \exp\left(-\frac{k\log \omega +c}{\omega}\cdot t/2\right)^v\\
&\le&\left[\frac{k\log \omega +c}{\omega}\right]^{v-1}\cdot \exp\left(-\frac{k\log \omega +c}{\omega}\cdot (1-\delta)\omega /2\right)^v\\
&=& \left[\frac{k\log \omega +c}{\omega}\right]^{v-1}\cdot \left[\omega^{-k\frac{1-\delta}{2}}\cdot e^{-c\frac{1-\delta}{2}}\right]^v.
\end{eqnarray*}
Combining with (\ref{three}) we get 
\begin{eqnarray*}
\binom t v v^{v-2} p_1^{v-1}q_1^{v(t-v)} &\le& (6\omega)^v \cdot  \left[\frac{k\log \omega +c}{\omega}\right]^{v-1}\cdot \omega^{-k\frac{v(1-\delta)}{2}} \cdot e^{-c\frac{v(1-\delta)}{2}}\\
&=& 6\cdot \left[6(k\log\omega +c)\right]^{v-1} \cdot \omega^{1-k\frac{v(1-\delta)}{2}}\cdot e^{-c\frac{v(1-\delta)}{2}}\\
&\le & \left\{  6\cdot \left[6(k\log\omega +c)\right]^{v-1} \omega^{-k\left[ \frac{v(1-\delta)}{2}-1   \right]}  \right\} \omega^{1-k} \cdot e^{-c}\\
&\le & \left\{  6\cdot \left[6(k\log\omega +c)\right]^{v-1} \omega^{-k\left[ \frac{v-1}{7}  \right]}  \right\} \cdot \omega^{1-k} \cdot e^{-c}\\
&=& 6\cdot \left\{\left(6(k\log\omega +c)\right)\cdot \omega^{-\frac{k}{7}}  \right\}^{v-1}\cdot \omega^{1-k} \cdot e^{-c}.
\end{eqnarray*}
On one of the steps we used the inequality $\frac{v(1-\delta)}{2} -1\ge \frac{v-1}{7}$.
Observe that the expression 
$$q= \left(6(k\log\omega +c)\right)\cdot \omega^{-\frac{k}{7}}$$ tends to $0$ as $n\to \infty$ and hence there exists $N$ such that for all $n>N$ one has
\begin{eqnarray*}
\sum_{v=3}^{t/2} \binom t v v^{v-2} p_1^{v-1}q_1^{v(t-v)} &\le& 6 \omega^{1-k}e^{-c}\left\{\sum_{v=3}^{t/2} q^{v-1}\right\} \\
&\le&
12 q^2\cdot \omega^{1-k} e^{-c}\\
&\le & \omega^{1-k}e^{-c}.
\end{eqnarray*}
Combining this inequality with the estimates for $v=1$ and $v=2$ completes the proof of Theorem \ref{propcon}, as explained above. \end{proof}

\section{When is a random simplicial complex simply connected?}

In this section we give establish a region of simple connectivity of the random complex $Y\in \Omega_n^r$ with respect to the probability measure 
$\PP_{r, \p}$ where $$\p=(p_0, p_1, \dots, p_r).$$ Recall that a simplicial complex $Y$ is said to be simply connected if it is connected and its fundamental group 
$\pi_1(Y, y_0)$ is trivial. The last condition is equivalent to the requirement that any continuous map of circle $S^1\to Y$ can be extended to a continuous map 
of the 2-disc $D^2 \to Y$. 

As in the previous section we shall assume that 
\begin{eqnarray}\label{everywhere1}
p_0=\frac{\omega}{n}, \quad \mbox{where}\quad \omega \to \infty.
\end{eqnarray}

\begin{theorem}\label{simplec}
Let $Y\in \Omega_n^r$ be a random complex with respect to the measure $\PP_{r, \p}$ where $\p=(p_0, \dots, p_r)$. 
Additionally to (\ref{everywhere1}) we shall assume that there exist sequences $\omega_1, \omega_2, \omega_3\to \infty$ one has
\begin{eqnarray}
\omega p_1^3 &=& 3 \log \omega +\omega_1, \label{three1}\\
\omega p_1^2p_2 &=& 2\log \omega +\omega_2, \label{three2}\\
\omega p_1^3 p_2^2 &= &3\log \omega +6\log p_1 + \omega_3.\label{three3}
\end{eqnarray}
Then $Y$ is simply connected a.a.s.
\end{theorem}

\begin{remark} {\rm In general, the conditions (\ref{everywhere1}), (\ref{three1}), (\ref{three2}), (\ref{three3}) are independent. For example, if 
$p_2=1$ then (\ref{three1}) implies (\ref{three2}) and (\ref{three3}), while if $p_1=1$ then (\ref{three1}) is satisfied automatically and (\ref{three3}) implies (\ref{three2}). 
However if we assume that $p_i=n^{-\alpha_i}$ where $\alpha_i\ge 0$ are constants then (\ref{three1}), (\ref{three2}), (\ref{three3}) become
\begin{eqnarray}
\alpha_0 +3\alpha_2 <1,\label{three4}\\
\alpha_0+ 2\alpha_1 + \alpha_2 <1,\label{three5} \\
\alpha_0+3\alpha_1+ 2\alpha_2<1. \label{three6}
\end{eqnarray}
and we see that the last inequality (\ref{three6}) implies the inequalities (\ref{three4}) and (\ref{three5}).
}
\end{remark} 

\begin{corollary} \label{cor73} Let $\p=(p_0, p_1, \dots, p_r)$ be a multi-parameter of the form $p_i=n^{-\alpha_i}$, where $\alpha_i$ are constants, 
$i=0, 1, \dots, r$. A random complex $Y\in \Omega_n^r$ is simply connected assuming that 
\begin{eqnarray}
\alpha_0+3\alpha_1+ 2\alpha_2<1. 
\end{eqnarray}
\end{corollary}

We shall show in a forthcoming paper that a random complex is not simply connected if $\alpha_0+3\alpha_1+ 2\alpha_2>1.$

\begin{remark} {\rm In the special case $p_0=p_1=1$ (the Linial - Meshulam model) Theorem \ref{simplec} reduces to Theorem 1.4 from \cite{BHK}. In the special case $p_0=p_2=1$ (clique complexes of random graphs) the result of Theorem \ref{simplec} follows from Theorem 3.4 from \cite{Kahle1}. 

The general plan of the proof of Theorem \ref{simplec} repeats the strategy of \cite{Kahle1}, proof of Theorem 3.4; namely, we apply the Nerve Lemma to the cover by stars of vertices. 
}
\end{remark}
%

First recall a version of the Nerve Lemma, see Lemma 1.2 in \cite{BLVZ}.

\begin{lemma}[The Nerve Lemma]\label{nerve}
 Let $X$ be a simplicial complex and let $\{S_i\}_{i\in I}$ be a family of subcomplexes covering $X$. 
 Suppose that for any $t\geq 1$ every non-empty intersection $$S_{i_1}\cap\ldots\cap S_{i_t}$$ is $(k-t+1)$-connected . 
 Then $X$ is $k$-connected if and only if the nerve complex
  $\mathcal{N}(\{S_i\}_{i\in I})$ is $k$-connected. 
\end{lemma}

Recall that the nerve $\mathcal{N}(\{S_i\}_{i\in I})$ is defined as the simplicial complex on the vertex set $I$ 
with a subset $\sigma\subset I$ forming a simplex if and only if the intersection
 $\cap_{i\in \sigma} S_i\neq\emptyset$ is not empty.

Given a random simplicial complex $Y\subset \Delta_n^{(r)}$, one may apply the Nerve Lemma \ref{nerve} to the cover 
$\{S_i\}_{i\in I}$, where $I=V(Y)$ (the set of vertices of $Y$) and $S_i$ is the star of the vertex $i$ in $Y$. Note that each star $S_i$ is contractible so that the condition of the Lemma \ref{nerve} is automatically satisfied for $t=1$. To establish the simple connectivity of $Y$ we need to show that (a) the intersection of any two stars $S_i\cap S_j$ is connected and (b) that the nerve complex $\mathcal{N}(\{S_i\}_{i\in I})$ is simply connected. 

Let us first tackle the task (b). The nerve $\mathcal{N}(\{S_i\}_{i\in I})$ is simply connected provided it has complete 2-dimensional skeleton, i.e. the intersection of any three stars $S_i\cap S_j\cap S_r\not=\emptyset$ is non-empty. This condition can be expressed by saying that any 3 vertices of $Y$ have a common neighbour, compare \cite{Kahle1}, \cite{Mesh}. 
The following Lemma describes the conditions under which any $k$ vertices of a random simplicial complex $Y\in \Omega_n^r$ have a common neighbour. 

\begin{lemma} \label{neighbour} Assume that a random simplicial complex $Y\in \Omega_{n}^r$ with respect to the measure $\PP_{r, \p}$ where $\p=(p_0, \dots, p_r)$, satisfies \begin{eqnarray}
\label{32}
p_0=\frac{\omega}{n}, \quad\quad  
\omega\to \infty\end{eqnarray} and 
\begin{eqnarray}\label{33}
p_1\, =\,  \left(\frac{k\log \omega +\omega_1}{\omega}\right)^{1/k}
\end{eqnarray}
where $k\ge 2$ is an integer and $\omega_1 \to \infty$. 
Then every $k$ vertices of $Y$ have a common neighbour, a.a.s.
\end{lemma}

\begin{proof} Given a subset $S\subset \{1, \dots,n\}$ with $|S|=k$ elements, we want to estimate the probability that a random complex 
$Y\in \Omega_n^r$ contains $S$ and the vertices of $S$ have no common neighbours in $Y$. 

Let $T\subset \{1, \dots, n\}$ be a set of $|T|=t$ vertices containing $S$ and 
let $$E_J=\{e_\alpha\}_{\alpha\in J}$$
be a set of edges of $\Delta_n$ such that each edge $e_\alpha$ connects a point $\alpha(0)\in S$ with a point $\alpha(1)\in T- S$ and 
for any $i\in T - S$ there exists $\alpha\in J$ such that $\alpha(1)=i$. Clearly, $t-k \le |J| \le k(t-k)$. 
Denote by $A_J\subset \Delta_n$ the graph obtained by adding to $T$ all edges connecting points of $S$ with points of $T-S$ which do not belong to $E_J$. 
Denote by $B_J$ the subcomplex of $\Delta_n$ obtained from the simplex $\Delta_T$ spanned by $T$ by removing the union of open stars of the edges 
$e_\alpha$, where $\alpha\in J$. In other words, to obtain $B_J$ we remove from $\Delta_T$ all simplexes which contain one of the edges $e_\alpha$ for $\alpha\in J$. 
The pair $A_J\subset B_J$ satisfies the condition of Lemma \ref{cont}; indeed, external faces of $B_J$ are vertices $\{1, \dots, n\}-T$ and the edges $e_\alpha$; all these are also external faces of $A_J$. Applying Lemma \ref{cont} we obtain
\begin{eqnarray}\label{ta}
\PP_{r, \p}(A_J\subset Y\subset B_J) = p_0^t\cdot p_1^{k(t-k)-|J|}\cdot q_0^{n-t}\cdot q_1^{|J|}.
\end{eqnarray}

Note that any complex $Y\in \Omega_n^r$ containing the set of vertices $S$ and such that there is no common neighbour for $S$ in $Y$ satisfies 
$$A_J\subset Y\subset B_J$$ for $T=V(Y)$ (the set of vertices of $Y$) and for a unique choice of the set of edges $E_J$ (it is the set of edges connecting points of 
$S$ with points of $T-S$ which do not belong to $Y$). 

For a set of edges $J$ as above and for a vertex $i\in T-S$ we denote by $\beta_i^J$ the number of edges $e_\alpha\in E_J$ such that $\alpha(1)=i$. 
Then 
$$1\le \beta^J_i\le k\quad\mbox{and}\quad  |J|=\sum_{i\in T_S} \beta_i^J.$$
There are 
$\binom n t \cdot \binom t k$ choices for the pair $S\subset T$ and there are 
$$\prod_{i=1}^{t-k}\binom k {\beta^J_i}$$ choices for the set $E_J$ with given vector 
$(\beta_1^J, \dots, \beta_{t-k}^J)$, and each $\beta^J_i$ can vary in the interval $\{1, \dots, k\}$. Hence we obtain that the probability that a random complex $Y\in \Omega_n^r$ 
has $k$ vertices without a common neighbour equals
\begin{eqnarray*}&{}& \sum_{t=k}^n \binom n t \cdot \binom t k \cdot \sum_{1\le \beta_i\le k} \prod_{i=1}^{t-k} \binom k {\beta_i}\cdot 
p_0^t p_1^{k({t-k})-\sum \beta_i} \cdot q_0^{n-t}\cdot q_1^{\sum \beta_i}\\ 
&=&  \binom n k \sum_{t=k}^n \binom {n-k}{t-k} \cdot p_0^t \cdot p_1^{k(t-k)} \cdot \left\{\left(1+\frac{q_1}{p_1}\right)^{k} -1\right\}^{t-k}\cdot q_0^{n-t}\\
&=&  \binom n k \sum_{t=k}^n \binom {n-k}{t-k} \cdot p_0^{t}\left(1-p_1^{k}\right)^{t-k}\cdot q_0^{n-t}
\\&=& p_0^k \binom n k \left(q_0+p_0(1- p_1^{k})\right)^{n-k}\\
&=&  p_0^k \binom n k \left(1-p_0p_1^k\right)^{n-k}. 
\end{eqnarray*}
Hence taking into account our assumptions (\ref{32}) and (\ref{33}) we obtain that the probability that a random complex $Y\in \Omega_n^r$ 
has $k$ vertices without a common neighbour is 
\begin{eqnarray*}
&{}&p_0^k \binom n k \left(1-p_0p_1^k\right)^{n-k}\\
&\le&  p_0^kn^k e^{-np_0p_1^k\frac{n-k}{n}}\\
&\le& \omega^k \cdot e^{(-k\log \omega - \omega_1)\frac{n-k}{n}}\\
&=& \omega^{\frac{k^2}{n}}\cdot e^{-\omega_1\frac{n-k}{n}}.
\end{eqnarray*}
The logarithm of the first factor $\omega^{\frac{k^2}{n}}$ tends to $0$ (as $\log \omega\le \log n$) and therefore the first factor tens to $1$, i.e. it is bounded, while the second factor $e^{-\omega_1\frac{n-k}{n}}$ clearly tends to zero. This complex the proof. 
\end{proof}

\begin{proof}[Proof of Theorem \ref{simplec}]

Let $A_n^r\subset \Omega_n^r$ denote the set of simplicial complexes $Y$ such that for any two vertices $i, j\in Y$ the intersection of their links 
$\lk_Y(i)\cap \lk_Y(j)$ is connected.

Let $B_n^r\subset \Omega_n^r$ denote the set of simplicial complexes $Y$ such that the degree of any edge $e\subset Y$ satisfies $\deg_Y e\ge 1$. 

Let $C_n^r\subset \Omega_n^r$ denote the set of simplicial complexes $Y$ such that any three vertices of $Y$ have a common neighbour. 

Let us show that $\PP_{r, \p}(A_n^r)\to 1$ as $n \to \infty$ under the assumption (\ref{three3}). Indeed, by Lemma \ref{linksintersection}, the intersection of two links is a multiparameter random simplicial complex with the multi-parameter $(p'_0, p'_1, \dots, p'_{r-1})$ where $$p'_i = p_ip_{i+1}^2.$$ Next we apply Theorem 
\ref{propcon} with $k=3$. Our assumption (\ref{three3}) is equivalent to 
$$p_1'= \frac{3\log \omega' + \omega_3}{\omega'} , \quad \mbox{where}\quad \omega'=np_0',$$
$p_0'=p_0p_1^2$, $p_1'=p_1p_2^2$ and $\omega= np_0$. 
By Theorem \ref{propcon}, the probability that the intersection of links of a given pair of vertices of $Y$ is disconnected is less or equal than 
$Ce^{-\omega_3}\omega^{-2}$, for a universal constant $C$. It follows that the expected number of pairs of vertices 
$i, j$ of 
$Y\in \Omega_n^r$ such that the intersection $\lk_Y(i)\cap \lk_Y(j)$ is disconnected is less or equal than 
$$\binom n 2 \cdot p_0^2 \cdot Ce^{-\omega_3} \omega^{-2}\le Ce^{-\omega_3},$$
which tends to $0$ with $n$. 

\begin{figure}[h]
\centering
\includegraphics[width=0.75\textwidth]{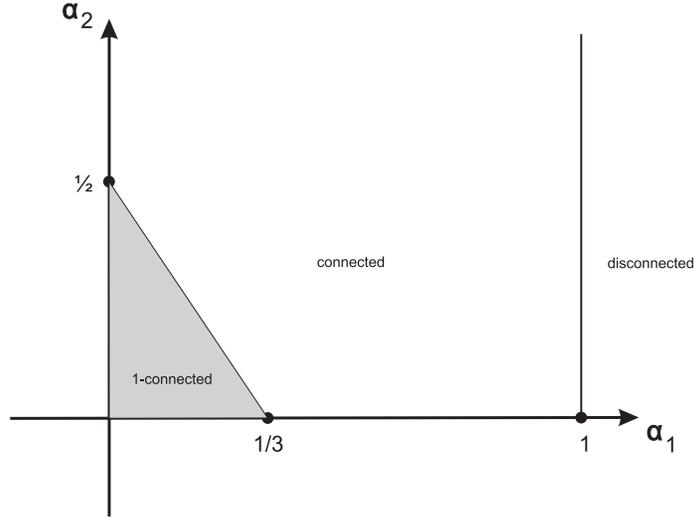}
\caption{Regions of connectivity and simple connectivity.} \label{simpleconnectivity}
\end{figure}
By Corollary \ref{degreezero}, the probability that $Y$ contains an edge of degree zero is less than $p_1e^{2-\omega_2}$ under the assumption (\ref{three2}). Hence we see that $\PP_{r, \p}(B_n^r)\to 1$ as $n \to \infty$, under the assumption (\ref{three2}). 

By Lemma \ref{neighbour}, due to the assumption (\ref{three1}), one has $\PP_{r, \p}(C_n^r)\to 1$ as $n \to \infty$. 

It follows that $$\PP_{r, \p}(A_n^r\cap B_n^r \cap C_n^r)\to 1$$ as $n \to \infty$. Let us show that every complex $Y\in A_n^r\cap B_n^r \cap C_n^r$ is simply connected. As explained in the paragraph preceding Lemma \ref{neighbour}, we may apply the Nerve Lemma \ref{nerve} to the cover by stars of vertices, and 
we only need to establish the task (a), i.e. to show that in a random complex $Y\in \Omega_n^r$ (under the assumptions of Theorem \ref{simplec}) the intersection of the stars of any two vertices is connected, a.a.s. Note that the task (b) is automatically satisfied because $Y\in C_n^r$. 

Let $i, j$ be two distinct vertices of $Y$. If the edge $(ij)$ is not contained in $Y$ then 
\begin{eqnarray}\st_Y(i)\cap \st_Y(j) = \lk_Y(i)\cap \lk_Y(j);\end{eqnarray}
This intersection is connected since $Y\in A_n^r$. 
However if $(ij)\subset Y$ then we have 
\begin{eqnarray}\label{inters}
\st_Y(i)\cap \st_Y(j) =\left(\lk_Y(i)\cap \lk_Y(j)\right)\cup \st_Y(ij).
\end{eqnarray}
The intersection $\lk_Y(i)\cap \lk_Y(j)$ is connected (since $Y\in A_n^r$) and $\lk_Y(i)\cap \lk_Y(j)$ is non-empty (since $Y\in B_n^r$). 
Since the star $\st_Y(ij)$ is non-empty and contractible, the union (\ref{inters}) is connected since 
 $$\lk_Y(i)\cap \lk_Y(j)\cap \st_Y(ij)\not=\emptyset$$
 (since $Y\in B_n^r$). 
 As explained above, the Nerve Lemma \ref{nerve} is applicable and implies that any $Y\in A_n^r\cap B_n^r \cap C_n^r$ is simply connected.
\end{proof}
\begin{figure}[h]
\centering
\includegraphics[width=0.6\textwidth]{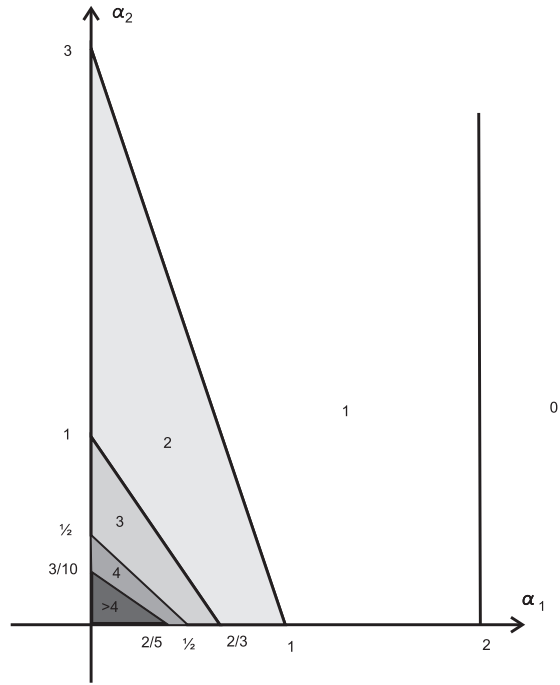}
\caption{Dimension of the random simplicial complex for various values of parameters $\alpha_1, \alpha_2$.}\label{dimension1}
\end{figure}
As an illustration consider the special case when $p_0=1$, $p_1=n^{-\alpha_1}$ and $p_2=n^{-\alpha_2}$ with $\alpha_1, \alpha_2$ being constant (i.e. independent of $n$). Then for $\alpha_1>1$ the random complex $Y$ is disconnected and for $\alpha_1<1$ the complex $Y$ is connected (see Example \ref{example63}) and for $3\alpha_1+2\alpha_2<1$ the complex $Y$ is simply connected (by Corollary \ref{cor73}). 
Figure \ref{simpleconnectivity} depicts the regions of connectivity. 

Figure \ref{dimension1} shows the dimension of a multi-parameter random simplicial complex, again assuming that 
$p_1=n^{-\alpha_1}$ and $p_2=n^{-\alpha_2}$ with $\alpha_1, \alpha_2$ being constant and $p_i=1$ for $i=0, 3, 4, \dots$. Details and proofs can be found in \cite{CF14}.  

In a forthcoming paper we shall show that in the domain $$1<3\alpha_1+2\alpha_2, \quad 0<\alpha_1<1, \quad 0<\alpha_2$$ the 
fundamental group of a random simplicial complex is nontrivial and hyperbolic in the sense of Gromov. It is a non-trivial random group depending on three parameters $p_0, p_1, p_2$ and we shall describe regions of various cohomological dimension and torsion in this random group.

\bibliographystyle{amsalpha}

\begin{thebibliography}{99}


\bibitem{BHK} E.\ Babson, C.\ Hoffman, M.\ Kahle, 
{\it The fundamental group of random $2$-complexes}, J. Amer. Math. Soc. 24 (2011), 1-28. 
See also the latest archive version arXiv:0711.2704 revised on 20.09.2012. 


\bibitem{BLVZ} A. Bj\"orner, L. Lov\'asz, S.T. Vre\'cica, R. \v Zivaljevi\'c, \textit{Chessboard complexes and matching complexes}. J. London Math. Soc. 49 (1994), no. 1, 25 -- 39.


\bibitem{CCFK} D. Cohen, A.E. Costa, M. Farber, T. Kappeler, \textit{Topology of random 2-complexes,} 
Journal of Discrete and Computational Geometry, {\bf 47}(2012), 117-149.

\bibitem{CFK} A. E. Costa, M. Farber, T. Kappeler, \textit{Topics of stochastic algebraic topology}. Proceedings of the Workshop on Geometric and Topological Methods in Computer Science (GETCO), 53 -- 70, Electron. Notes Theor. Comput. Sci., 283, Elsevier Sci. B. V., Amsterdam, 2012.


\bibitem{CF1} A.E. Costa, M. Farber, \textit{Geometry and topology of random 2-complexes},  arXiv:1307.3614, to appear in \textit{Israel Journal of Mathematics}. 

\bibitem{CFH} A.E. Costa, M. Farber, D. Horak, \textit{Fundamental groups of clique complexes of random graphs}, to appear in \textit{Transactions of the London Mathematical Society}.  

\bibitem{CF14} A.E. Costa, M. Farber, \textit{Random simplicial complexes}, arXiv:1412.5805. 

\bibitem{CF15a} A.E. Costa, M. Farber, \textit{Homological domination in large random simplicial complexes}, arXiv:1503.03253. 
\bibitem{DT}  N. Dunfield and W. P. Thurston, \textit{Finite covers of random 3-manifolds}, Invent.
Math. 166 (2006), no. 3, 457 -- 521. 

\bibitem{ER} P.\ Erd\H{o}s, A.\ R\'enyi, {\it On the evolution of 
random graphs}, Publ.\ Math.\ Inst.\ Hungar.\ Acad.\ Sci.\ {\bf 5}
(1960), 17--61.


\bibitem{F} M. Farber, \textit{Topology of random linkages}, Algebraic and Geometric Topology, 8(2008), 155 - 171.



\bibitem{JLR} S. Janson, T. {\L}uczak, A. Ruci\'nski, \textit{Random graphs}, Wiley-Intersci. Ser. Discrete Math. Optim., Wiley-Interscience, New York, 2000.

\bibitem {Kahle1} M. Kahle, Topology of random clique complexes, Discrete Math. 309 (2009), no. 6,
1658 -- 1671. 
\bibitem{Ksurvey} M. Kahle, \textit{Topology of random simplicial complexes: a survey},  To appear in AMS Contemporary Volumes in Mathematics. Nov 2014. arXiv:1301.7165.



\bibitem{LM} N.\ Linial, R.\ Meshulam, {\it Homological connectivity  of random $2$-complexes}, Combinatorica {\bf 26} (2006),  475--487.

\bibitem{Mesh} R. Meshulam, \textit{The clique complex and hypergraph matching}, Combinatorica {\bf 21} (2001), no. 1, 89 -- 94. 

\bibitem{MW} R.\ Meshulam, N.\ Wallach, {\it Homological connectivity of random $k$-complexes}, Random Structures \& Algorithms \textbf{34} (2009), 408--417. 



\bibitem{PS} N. Pippenger and K. Schleich, \textit{Topological characteristics of random triangu-
lated surfaces}, Random Structures Algorithms 28 (2006), no. 3, 247-288.

\end{thebibliography}

\end{document}